\documentclass[]{article}

\usepackage{times}
\usepackage{amsthm}
\usepackage{amsmath}
\usepackage{amssymb}
\usepackage{mathrsfs}
\usepackage{pb-diagram}
\usepackage{xcolor}

\newcommand{\A}{\mathcal{A}}

\newcommand{\B}{\mathcal{B}}

\newcommand{\h}{\mathcal{H}}

\newcommand{\rtt}{\rightthreetimes}

\def\o{\omega}

\def\te{\theta}

\def\N{\mathbb{N}}
\def\T{\mathbb{T}}
\def\Z{\mathbb{Z}}

\def\h{\mathcal H}
\def\K{\mathcal K}

\def\Z{\mathbb Z}
\def\pe{{\sf P}}

\def\e{{\sf e}}

\def\bu{\bullet}

\def\({\left(}
\def\[{\left[}
\def\){\right)}
\def\]{\right]}

\def\Si{\Sigma}
\def\G{{\sf G}}
\def\n{{\sf N}}
\def\wG{\widehat\G}
\def\p{\parallel}
\def\<{\langle}
\def\>{\rangle}

\usepackage{tikz-cd} 

 \newtheorem{thm}{Theorem}[section]
 \newtheorem{cor}[thm]{Corollary}

 \newtheorem{prop}[thm]{Proposition}
 \theoremstyle{definition}
 \newtheorem{defn}[thm]{Definition}
 \theoremstyle{remark}
 \newtheorem{rem}[thm]{Remark}
 \newtheorem{ex}[thm]{Example}
 \numberwithin{equation}{section}

\numberwithin{equation}{section}

\begin{document}


\title{Symmetry and Spectral Invariance for Topologically\\ Graded $C^*$-Algebras and Partial Action Systems}

\author{D. Jaur\'e and M. M\u antoiu}

\author{D. Jaur\'e and M. M\u antoiu
\footnote{
\textbf{2020 Mathematics Subject Classification:} Primary 43A20, Secondary 47L65,47L30.
\newline
\textbf{Key Words:} partial group action, symmetric Banach algebra, inverse closed, Wiener-Hopf operator, ordered group, graph algebra. 
\newline
{D. J.  acknowledges financial support from "Beca de Doctorado Nacional Conicyt". M. M. has been supported by the Fondecyt Project 1200884. 
}}
}



\maketitle


\begin{abstract}
A discrete group $\G$ is called rigidly symmetric if the projective tensor product between the convolution algebra $\ell^1(\G)$ and any $C^*$-algebra $\A$ is symmetric. We show that in each topologically graded $C^*$-algebra over a rigidly symmetric group there is a $\ell^1$-type symmetric Banach $^*$-algebra, which is inverse closed in the $C^*$-algebra. This includes new general classes, as algebras admitting dual actions and partial crossed products. Results including convolution dominated kernels, inverse closedness with respect with ideals or weighted versions of the $\ell^1$-decay are included. Various concrete examples are presented.
\end{abstract}

\tableofcontents

\section{Introduction}\label{introduction}

Spectral invariance and symmetry of normed algebras have been important concepts, starting with the classical article \cite{Wie}. Some general references which are related with the point of view of this article include \cite{FL,FGL,GL1,GL2,LP,Pa2,Po}. We also refer to \cite{Gr} for a very readable general presentation, making connections with related mathematical topics.

\smallskip
The basic definition is the following (we restrict to unital algebras; this will be enough in examples):

\begin{defn} \label{symmetric}
\begin{enumerate}
\item[(i)]
The unital $^*$-subalgebra $\mathfrak B$ of the $C^*$-algebra $\mathfrak C$ is called {\it inverse closed} (or {\it spectrally invariant}) if for each $\Phi\in\mathfrak B$ which is invertible in $\mathfrak C$\,, one actually has $\Phi^{-1}\!\in\mathfrak B$\,. 
\item[(ii)]
A unital Banach $^*$-algebra $\mathfrak B$ is called {\it symmetric} if the spectrum of $b^*b$ is positive for every $b\in\mathfrak B$ (this happens if and only if the spectrum of any self-adjoint element is real.)
\end{enumerate}
\end{defn}

The two notions are related, but in this Introduction we only discuss symmetry.
A great interest lies in algebras connected somehow to a discrete group. The following notions will be relevant in the article; the first two are classical, while the third constitutes our main concern:

\begin{defn} \label{groupsymmetric}
\begin{enumerate}
\item[(i)]
The discrete group $\G$ is called {\it symmetric} if the convolution Banach $^*$-algebra $\ell^1(\G)$ is symmetric.
\item[(ii)]
The discrete group $\G$ is called {\it rigidly symmetric} if given any $C^*$-algebra $\A$\,, the projective tensor product $\ell^1(\G)\otimes\A$ is symmetric.
\item[(iii)]
The discrete group $\G$ is called {\it hypersymmetric} if given any topologically graded $C^*$-algebra $\mathfrak C$\,, the Banach $^*$-algebra $\ell^1(\mathfrak C)$ (to be defined below) is symmetric.
\end{enumerate}
\end{defn}

The class of topologically graded $C^*$-algebras, described for example in \cite{Ex2} and reviewed briefly in Subsection \ref{flocinor},  is much larger than the one appearing in (ii). Simplifying, the main result of the article is

\begin{thm}\label{principala}
If a discrete group is rigidly symmetric, then it is hypersymmetric.
\end{thm}

\begin{rem}\label{lacantors}
Clearly, a rigidly symmetric group is symmetric. It is still not known if the two notions are really different. Anyhow, as a consequence of \cite[Thm.\,1]{Ku}, if $\G$ is supposed only symmetric, but the $C^*$-algebra $\A$ is type $I$, the projective tensor product $\ell^1(\G)\otimes\A$ is symmetric. Classes of rigidly symmetric discrete groups are (cf. \cite{LP}): (a) Abelian, (b) finite, (c) finite extensions of discrete nilpotent. This last class includes all the finitely generated groups with polynomial growth. A central extension of a rigidly symmetric group is rigidly symmetric, by \cite[Thm.\,7]{LP}. In \cite[Cor.\,2.16]{Ma} it is shown that the quotient of a discrete rigidly symmetric group by a normal subgroup is rigidly symmetric. 
\end{rem}

There are intermediary notions between rigid symmetry and hypersymmetry. It has already been proved \cite{FGL,BB,Ma} that if a discrete group is rigidly symmetric, it is also symmetric in the sense of global crossed products: For an action $\beta$ of $\G$ in a $C^*$-algebra $\A$\,, one canonically constructs a crossed product $C^*$-algebra $\A\!\rtimes_\beta\!\G$\,, which is the enveloping algebra of a Banach $^*$-algebra $\ell^1_\beta(\G;\A)$\,. If for every such global action $(\G,\beta,\A)$ the algebra $\ell^1_\beta(\G;\A)$ is symmetric, $\G$ may be called {\it symmetric in the sense of global crossed products}. For a trivial action $\beta$ one gets the projective tensor product $\ell^1(\G)\otimes\A\cong\ell^1(\G;\A)$ and plenty of interesting $C^*$-algebras proved to be global crossed products for some non-trivial action, so \cite{FGL} supplied a lot of new examples of symmetric algebras. 

\smallskip
Our initial aim was to extend this further to {\it partial actions} $\te$ of discrete groups $\G$ on $C^*$-algebras $\A$\,. The class of $C^*$-algebras that may be realized as partial crossed products $\A\!\rtt_\te\!\G$\,, but which are not accessible to global actions, is quite impressive. We soon understood that topologically graded algebras over the discrete group $\G$ provide an even larger setting, which is also simpler and more natural in a certain sense. So Theorem \ref{teoremix} and a consequence on covariant convolution dominated kernels are stated in this framework. 

\smallskip
However, since the subcase of partial crossed products  is very important, and many $C^*$-algebras are proved to be of such a type (a non-trivial task, very often), in Subsection \ref{cerculix} we made our results explicit for this situation. In the same Subsection \ref{cerculix}, we indicated the connection between topologically $\G$-graded $C^*$-algebras $\mathfrak C$ and "dual actions" of $\G$. The right non-commutative concept is that of a coaction of $\G$ on $\mathfrak C$\,, but for simplicity and having in mind the examples we want to treat by this tool, we only discussed the case of Abelian groups. In this case one works with a (global) action of the compact Abelian Pontryagin dual $\wG$ and the spaces of the grading are defined to be the spectral subspaces of this dual action. 

\begin{rem}\label{lantors}
In \cite[Cor.\,4.8]{SW} it is shown that a symmetric group is amenable. Since our results below require $\G$ to be rigidly symmetric, {\it we will assume $\G$ to be (at least) amenable}. This will simplify the grading theory. In particular, the (universal) partial crossed product $\A\rtt_\te\!\G$ defined in Subsection \ref{cerculix} coincides with the reduced one \cite[Sect.\,3]{MC}.
\end{rem}

In Subsection \ref{idealuri} we treat inverse closednes with respect to a suitable ideal. The most important application is to Fredholm properties of operators, the ideal being (isomorphic to) the ideal of all the compact operators in a Hilbert space. 

\smallskip
In Subsection \ref{decay} we explore others types of decay, using weighted $\ell^1$-algebras. It is possible to implement the result in all the examples, but we will not do it explicitly.

\smallskip
Very often one encounters $C^*$-algebras that are not obviously involving a grading or a partial action of a group. Showing that such a structure can be put into evidence is a valuable tool to study various properties, in particular the existence of inverse closed subalgebras. In our last Section, using various references to the literature, we present such examples and find out what the $\ell^1$-algebra is composed of in each case. Our treatment is not meant to be exhaustive, but it will be clear that many new interesting cases are covered. In addition, some examples belong to different classes or can be treated from several points of view; we stick to the one that seems to be simpler. When the group is Abelian, we mostly use dual actions. In other cases partial actions or a direct presentation of the topological grading is the tool.

\smallskip
Some conventions and notations: $C^*$-algebraic morphisms and representations are assumed to be involutive, the ideals are self-adjoint, bi-sided and closed. $\mathbb B(\h)$ will denote the $C^*$-algebra of linear bounded operators in a Hilbert space; $\mathbb K(\h)$ is the ideal of compact operators.

\section{The abstract theory}\label{astrakt}

\subsection{The main result}\label{flocinor}

All over the article, $\G$ will be a discrete group with unit $\e$\,. As explained in Remark \ref{lantors}, there is no loss if we assume it amenable. For constructions involving graded $C^*$-algebras we followed \cite{Ex3,FD2}\,; we decided not to mention Fell bundles explicitly (with one exception), but they can be perceived in the background.

\begin{defn}\label{claro}
We say that the unital $C^*$-algebra $\mathfrak C$ is {\it $\G$-graded} if that there is a family $\big\{\mathfrak C_g\,\vert\,g\in\G\big\}$ of closed subspaces of $\mathfrak C$ such that the algebraic direct sum $\,\bigoplus_{g\in\G}\mathfrak C_g$ is dense in $\mathfrak C$ and such that
\begin{enumerate}
\item[(i)] 
$\mathfrak C_g\mathfrak C_h\subset\mathfrak C_{gh}\,$ for every $g,h\in\G$\,, 
\item[(ii)] 
$\mathfrak C_g^*\subset\mathfrak C_{g^{-1}}\,$ for every $g\in\G$\,,
\end{enumerate}
For every $h\in\G$ we denote by $P_h:\bigoplus_{g\in\G}\!\mathfrak C_g\to\mathfrak C_h$ the canonical projection (in general it might not be continuous). If $P_\e$ is continuous, $\mathfrak C$ is {\it a topologically $\G$-graded $C^*$-algebra.} We are going to write
\begin{equation}\label{veverita}
\mathfrak C=\widetilde{\bigoplus}_{g\in\G}\mathfrak C_g\,.
\end{equation}
\end{defn}

For $\Phi,\Psi\in\bigoplus_g\mathfrak C_g$ and $g\in\G$ one verifies easily that
\begin{equation}\label{multiplix}
P_g(\Phi\Psi)=\sum_{hk=g}P_h(\Phi)P_k(\Psi)\,,
\end{equation}
\begin{equation}\label{multiplex}
P_g(\Phi)^*=P_{g^{-1}}(\Phi^*)\,.
\end{equation}
When the grading is topological, there are several facilities which will be crucial for our approach. Let us denote by $\widetilde P_\e:\mathfrak C\to\mathfrak C_\e$ the extension of $P_\e$ to a linear bounded map. By Tomiyama's Theorem (\cite[Thm.\ 1.5.10]{BO}, see also \cite[Thm.\,3.3]{Ex3}), it is a positive contractive conditional expectation. Then all the projections $P_g$ do extend to contractions $\widetilde P_g:\mathfrak C\to\mathfrak C_g$\,. We refer to the proof of Theorem \ref{teoremix},(ii) for another use of topological grading, in conjunction with amenability.

\begin{defn}\label{glaro}
Let $\mathfrak C$ be a topologically graded $C^*$-algebra. On $\,\bigoplus_{g}\!\mathfrak C_g$  we can introduce the new norm
\begin{equation*}\label{normix}
\p\!\Phi\!\p_{\ell^1(\mathfrak C)}\,:=\sum_{g\in\G}\p\! P_g(\Phi)\!\p.
\end{equation*}
The completion $\ell^1(\mathfrak C)$ of $\bigoplus_{g}\!\mathfrak C_g$ in this norm is called {\it the $\ell^1$-algebra of the graded $C^*$-algebra $\mathfrak C$}\,.
\end{defn} 

Since $\p\!\cdot\!\p\,\le\,\p\!\cdot\!\p_{\ell^1(\mathfrak C)}$ on $\bigoplus_{g}\!\mathfrak C_g$\,, one may interpret $\ell^1(\mathfrak C)$ as a subspace of $\mathfrak C$\,:
\begin{equation*}\label{interpretatix}
\ell^1(\mathfrak C)=\Big\{\Phi\in\mathfrak C\;\Big\vert\,\sum_{g\in\G}\p\! \widetilde P_g(\Phi)\!\p\,<\infty\Big\}\,.
\end{equation*} 
In fact it is a Banach $^*$-subalgebra with the algebraic structure borrowed from $\mathfrak C$ and its new norm. Its elements may be written as unconditionally convergent series $\Phi=\sum_{g\in\G}\widetilde P_g(\Phi)$ in the norm topology.

\smallskip
The next result, besides being interesting in itself, is basic for our approach.

\begin{prop} \label{isometrix}
Let $\mathfrak C=\widetilde{\bigoplus}_{g\in\G}\mathfrak C_g$ be a topologically  graded $C^*$-algebra over a discrete group $\G$\,. There exists a $C^*$-algebra $\B$ and an isometric $^*$-morphism
\begin{equation*}\label{guyana}
T:\ell^1(\mathfrak C)\hookrightarrow \ell^1(\G;\mathcal B)\equiv\ell^1(\G)\otimes \B\,.
\end{equation*}
\end{prop}

\begin{proof}
Let $\pi:\mathfrak C\to\mathbb B(\h)$ be a faithful representation. For any $g\in\G$ we denote by $\pi_g$ the restriction of $\pi$ to the closed subspace $\mathfrak C_g$\,; it is a linear isometry. One has
\begin{equation}\label{primika}
\pi_g(\Phi)\pi_h(\Psi)=\pi_{gh}(\Phi\Psi)\,,\quad\forall\,g,h\in\G\,,\ \Phi\in\mathfrak C_g\,,\ \Psi\in\mathfrak C_h\,,
\end{equation}
\begin{equation}\label{secundika}
\pi_g(\Phi^*)=\pi_{g^{-1}}(\Phi)^*,\quad\forall\,g\in\G\,,\ \Phi\in\mathfrak C_g\,.
\end{equation}

Then we set
\begin{equation}\label{tronc}
\big[T(\Phi)\big](g):=\pi_g\big[\widetilde P_g(\Phi)\big]\,,\quad\Phi\in\ell^1(\mathfrak C)\,,\ g\in\G\,.
\end{equation}
Clearly $T$ is a well-defined linear isometry from $\ell^1(\mathfrak C)$ to $\ell^1\big(\G,\mathbb B(\h)\big)$\,:
$$
\p\!T(\Phi)\!\p_{\ell^1(\G,\mathbb B(\h))}\,=\sum_{g\in\G}\p\!\pi_g\big[\widetilde P_g(\Phi)\big]\!\p_{\mathbb B(\h)}\,=\sum_{g\in\G}\p\!\widetilde P_g(\Phi)\!\p\,=\,\p\!\Phi\!\p_{\ell^1(\mathfrak C)}.
$$
To check that $T$ respects the algebraic operations, it is enough to work with elements of the algebraic direct sum. By taking $\Phi,\Psi\in\bigoplus_g\!\mathfrak C_g$\,, one can write 
$$
\begin{aligned}
\big[(T\Phi)\star(T\Psi)\big](g)&=\sum_{h\in \G}(T\Phi)(h)(T\Psi)(h^{-1}g) \\
&=\sum_{h\in \G}\pi_h\left[P_h(\Phi)\right]\,\pi_{h^{-1}g}\big[P_{h^{-1}g}(\Psi)\big] \\
&\overset{\eqref{primika}}{=}\sum_{h\in \G}\pi_g\big[P_h(\Phi)P_{h^{-1}g}(\Psi)\big]\\
&=\pi_g\Big(\sum_{h\in \G}P_h(\Phi)P_{h^{-1}g}(\Psi)\Big)\\
&\overset{\eqref{multiplix}}{=}\pi_g\big[P_g(\Phi\Psi)\big]=\big[T(\Phi\Psi)\big](g)\,, 
\end{aligned}
$$
showing that $T$ is an algebraic morphism. Finally we treat the involution:
$$
\begin{aligned}
\left[T(\Phi)\right]^\star\!(g) & =\pi_{g^{-1}}\!\left[P_{g^{-1}}(\Phi)\right]^*\overset{\eqref{secundika}}{=} \pi_g\big[P_{g^{-1}}(\Phi)^*\big]\\
& \overset{\eqref{multiplix}}{=} \pi_g[P_g(\Phi^*)]= \big[T(\Phi^*)\big](g)\,.
\end{aligned}
$$
\end{proof}

With Proposition \ref{isometrix} in hand, we prove now our main abstract result.  Recall that a Banach $^*$-algebra is called {\it reduced} if its universal $C^*$-seminorm is in fact a norm. In this case, its enveloping $C^*$-algebra is simply its completion in this $C^*$-norm (no quotient is needed). 

\begin{thm}\label{teoremix}
Let $\mathfrak C=\widetilde{\bigoplus}_{g\in\G}\mathfrak C_g$ be topologically graded over a  rigidly symmetric discrete group $\G$\,.
\begin{enumerate}
\item[(i)] 
$\ell^1(\mathfrak C)$ is a symmetric Banach $^*$-algebra.
\item[(ii)] 
$\ell^1(\mathfrak C)$ is an inverse closed subalgebra of $\,\mathfrak C$\,.
\item[(iii)] 
Let $\Pi:\mathfrak C\to \mathbb{B}(\h)$ be a faithful representation. Then $\Pi\left[\ell^1(\mathfrak C)\right]$ is inverse closed in $\mathbb{B}(\h)$\,.
\end{enumerate}
\end{thm}

\begin{proof}
(i) By \cite[Th.11.4.2]{Pa2}, a closed $^*$-algebra of a symmetric Banach $^*$-algebra is also symmetric. Proposition \ref{isometrix} and the fact that $\ell^1(\G;\mathcal B)$  was assumed  symmetric proves the result.

\smallskip
(ii) It is known \cite[11.4]{Pa2} that a reduced Banach $^*$-algebra is symmetric if and only if it is inverse closed in its enveloping $C^*$-algebra. If we show that we are in such a case, we then apply (i). Our $\ell^1(\mathfrak C)$ is clearly reduced, being given as a $^*$-subalgebra of $\mathfrak C$\,. Actually, we claim that $\mathfrak C$ can be seen as (a copy of) the enveloping $C^*$-algebra of $\ell^1(\mathfrak C)$\,. In \cite[Corol.\,4.8]{SW} it is shown that a symmetric group (as our rigidly symmetric $\G$) is surely amenable. For such groups, one may identify $\mathfrak C$ with both the full and the reduced $C^*$-algebras of the Fell bundle associated to the topological grading \cite[Prop.\,4.2\,,Thm.\,4.7]{Ex3} and then the assertion about enveloping follows \cite[VIII.17]{FD2}.

\smallskip
(iii) Is a consequence of (ii), since any $C^*$-algebra is inverse closed in a larger $C^*$-algebra.
\end{proof}

In the context of group actions on $C^*$-algebras one deals with convolution dominated kernels and operators \cite{FL,FGL,Gr,GL2,Ma}. We briefly indicate an extension for the topologically graded case, leading to another form of a symmetric Banach $^*$-algebra at the level of kernels. Representations by integral operators and their inverse closedness properties are left to the reader.

\begin{defn}\label{kimpembe}
{\it A convolution dominated kernel} (or {\it matrix}) is a function $K:\G\times\G\to \mathfrak C$ such that $K(g,h)\in\mathfrak C_{gh^{-1}}$ for every $g,h\in\G$ and such that the norm
\begin{equation*}\label{cambio}
\p\!K\!\p_\mathscr K\,:=\inf\big\{\!\p\!\kappa\!\p_{\ell^1(\G)}\,\big\vert\,\p\!K(g,h)\!\p_\mathfrak C\,\le|\kappa(gh^{-1})|\,,\,\forall\,g,h\in\G\big\}
\end{equation*}
is finite. We denote by $\mathscr K_\G(\mathfrak C)$ the vector space of all these convolution dominated kernels. An element $K$ of $\mathscr K_\G(\mathfrak C)$ is called {\it covariant}, and we write $K\in\mathscr K^{\rm co}_\G(\mathfrak C)$\,, if
\begin{equation*}\label{pusimic}
K(gk,hk)=K(g,h)\,,\quad\forall\,g,h,k\in\G\,.
\end{equation*}
\end{defn}

It is straightforward to show that endowed with the norm $\p\!\cdot\!\p_\mathscr K$\,, the multiplication
\begin{equation*}\label{greta}
(K\bu L)(g,h):=\sum_{k\in\G}K(g,k)L(k,h)
\end{equation*}
and the involution $K^\bu(g,h):=K(h,g)^*$, the space $\mathscr K_G(\mathfrak C)$ is a Banch $^*$-algebra. 

\begin{cor}\label{climatic}
The subspace $\mathscr K^{\rm co}_\G(\mathfrak C)$ is a symmetric Banach $^*$-algebra.
\end{cor}

\begin{proof}
It is easy to check that $\mathscr K^{\rm co}_\G(\mathfrak C)$ is a closed $^*$-subalgebra of $\mathscr K_\G(\mathfrak C)$\,.
Let us define 
\begin{equation*}\label{london}
\big[\Upsilon(\Phi)\big](g,h):=\widetilde P_{gh^{-1}}(\Phi)\,,\quad\forall\,g,h\in\G\,.
\end{equation*}
Then $\Upsilon:\ell^1(\mathfrak C)\to\mathscr K_\G(\mathfrak C)$ is an isometric $^*$-morphism with range $\mathscr K^{\rm co}_\G(\mathfrak C)$\,. For the algebraic relations use \eqref{multiplix} and \eqref{multiplex} (extended). On this range, the inverse reads
\begin{equation*}\label{reads}
\big[\Upsilon^{-1}(K)\big]_g:=K(g,\e)\,,\quad\forall\,g\in\G\,.
\end{equation*}
Since $\mathscr K^{\rm co}_\G(\mathfrak C)$ and $\ell^1(\mathfrak C)$ are isomorphic, the result follows from Theorem \ref{teoremix}\;(i).
\end{proof}

\begin{rem}\label{fratrix}
In \cite{Ra} it is shown how to deform Fell bundles by a {\it $2$-cocycle} $\o:\G\!\times\!\G\to\T$\,, bound to satisfy for every $g,h,k\in\G$ the axioms
\begin{equation*}\label{bundix}
\o(g,h)\o(gh,k)=\o(h,k)\o(g,hk)\,,\quad\o(g,\e)=1=\o(\e,g)\,.
\end{equation*}
Starting with a topologically graded $C^*$-algebra $\mathfrak C=\widetilde{\bigoplus}_{g\in\G}\mathfrak C_g$\,, it is possible to construct a new one $\mathfrak C\{\o\}=\widetilde{\bigoplus}^\o_{g\in\G}\mathfrak C_g$\,. {\it The Banach $^*$-algebra $\ell^1\big(\mathfrak C\{\o\}\big)$ is different from $\ell^1(\mathfrak C)$ and it is symmetric and inverse closed in $\mathfrak C\{\o\}$\,.}
\end{rem}

\subsection{Topologically graded algebras from dual actions and partial crossed products}\label{cerculix}

Given a $C^*$-algebra, it is often not obvious whether it has an interesting grading or not. In this subsection we are going to review some ways to solve this problem, following especially results of Exel, Quigg and Raeburn. This will lead to two reformulations of Theorem \ref{theoremix} that are useful when studying examples and which will also provide the reader with a more extended picture. We will take advantage of the fact that our group $\G$ is discrete and amenable; this will allow for some simplifications.

\smallskip
The most general statement is that in this setting, {\it topologically graded $C^*$-algebras over $\G$ are equivalent to coactions of $\,\G$}\,, cf. \cite{Qu} and \cite[App.\,A]{Ra}. In this direction we only indicate the case of Abelian groups, that will be enough for the examples we intend to treat via this approach.

\smallskip
So we fix for a while a discrete Abelian group $\G$\,. Its Pontryagin dual $\widehat\G$ is a compact Abelian group with normalized Haar measure $d\chi$\,. Let $\big(\mathfrak C,\alpha,\wG\,\big)$ be a (usual, full) action of $\wG$ on a $C^*$-algebra. For every $g\in\G$ let
\begin{equation*}\label{spectrix}
\mathfrak C_g:=\big\{\Phi\in\mathfrak C\,\big\vert\,\alpha_\chi(\Phi)=\chi(g)\Phi\,,\,\forall\,\chi\in\wG\big\}
\end{equation*}
be the $g$'th spectral subspace of the action. It is easy to see that one gets a grading, as in \eqref{veverita}.
We have now the concrete projections
\begin{equation*}\label{projetrix}
\widetilde P_g:\mathfrak C\to\mathfrak C_g\,,\quad \widetilde P_g(\Phi):=\int_{\wG}\,\overline{\chi(g)}\,\alpha_\chi(\Phi)\,d\chi\,,
\end{equation*}
which are obviously contractive; the grading is topological. 
We can reformulate Theorem \ref{teoremix} in this Abelian setting; the $\ell^1$-algebra will deserve a special notation, to recall the action $\alpha$ and the type of decay.

\begin{thm}\label{exelix}
Let $\big(\mathfrak C,\alpha,\wG\,\big)$ be a continuous action of a compact Abelian group $\wG$ on a unital $C^*$-algebra. Denoting by $\G$ the dual of $\,\wG$\,,  
\begin{equation*}\label{peru}
\ell^1(\mathfrak C)\equiv\bigoplus_{g\in\G}^{1,\alpha}\mathfrak C_g=\Big\{\Phi\in\mathfrak C\;\Big\vert\,\sum_{g\in\G}\p\! \widetilde P_g(\Phi)\!\p\,<\infty\Big\}
\end{equation*}
is a reduced and symmetric Banach $^*$-algebra, which is inverse closed in $\mathfrak C$\,.
\end{thm}

\begin{rem}\label{crucix}
Using normal subgroups of $\G$ one can increase the family of symmetric subalgebras. If, as above, $\big(\mathfrak C,\alpha,\wG\,\big)$ is a continuous action of the dual of the Abelian discrete group $\G$ and $\n$ is a normal subgroup, the dual $\widehat{\G/\n}$ of the quotient is isomorphic to the closed subgroup
\begin{equation}\label{clarix}
{\rm Ann\big(\n\!\mid\!\wG\big)}:=\big\{\chi\in\wG\,\big\vert\,\chi|_\n=1\big\}\,.
\end{equation}
Thus from $\alpha$ we deduce a continuous action $\alpha^\n$ of $\widehat{\G/\n}$ on $\mathfrak C$\,, first by restricting $\alpha$ to ${\rm Ann\big(\n\!\mid\!\wG\big)}$ and then composing with the isomorphism. In this way we obtain the new {\it symmetric} Banach $^*$-algebra
\begin{equation*}\label{perux}
\ell^1\big(\mathfrak C^{\G/\n}\big)\equiv\bigoplus_{\gamma\in\G/\n}^{1,\alpha^\n}\mathfrak C_\gamma\,.
\end{equation*}
Extreme cases are $\ell^1\big(\mathfrak C^{\G/\{\e\}}\big)=\ell^1(\mathfrak C)$ and $\ell^1\big(\mathfrak C^{\G/\G\}}\big)=\mathfrak C$\,.
\end{rem}

Crossed products assigned to partial actions of (discrete) groups $\G$ on (unital) $C^*$-algebras are among the most important examples of graded $C^*$-algebras, and they were our initial motivation. In \cite[Thm.\,27.11]{Ex2} they are characterized among the larger class, while \cite[Thm.\,4.1]{QR} gives a Landstad-type description of them in terms of the associated coaction. In many cases what we are actually given is the partial action; the grading comes afterwards. Now the discrete group $\G$ is no longer Abelian.

\begin{defn}\label{bazika}
A {\it partial action} of $\G$ on the $C^*$-algebra $\A$ is a pair $\big(\{\A_g\}_{g\in \G},\{\te_g\}_{g\in\G}\}\big)$\,,
where for each $g\in \G$\,, $\A_g$ is a closed two-sided ideal in $\A$\,, $\te_g$ is a $^*$-isomorphism from $\A_{g^{-1}}$ onto $\A_g$\,, satisfying the following postulates, for all $g,h\in\G$\,:
\begin{enumerate}
\item[(i)] $\A_\e=\A$ and $\te_\e$ is the identity automorphism of $\A$\,,
\item[(ii)] $\te_g\!\left(\A_{g^{-1}}\!\cap \A_h\right)\subset\A_{gh}$\,,
\item[(iii)] $\te_g\big(\te_h(a)\big)=\te_{gh}(a)$\,, $\forall\,a\in\A_{h^{-1}}\!\cap\A_{(gh)^{-1}}$\,.
\end{enumerate}
\end{defn}

We also denote by $(\A,\te,\G)$ the partial dynamical system. It follows easily that we have $\te_{g^{-1}}\!=\te_g^{-1}$.
Our main object is the following vector space:
\begin{equation*}\label{thetika}
\ell^1_\te(\G;\A):= \left\{\Phi\in\ell^1({\sf G};\A)\mid\Phi(g)\in \A_g \,,\,\forall\,g\in \G\right\},
\end{equation*}
which is obviously closed. For $\Phi,\Psi\in\ell^1_\te(\G;\A)$ we define the product
\begin{equation*}\label{strugure}
\left( \Phi\star_{\theta} \Psi\right)(g)  =\sum_{h\in \G} \theta_h\!\left[\theta_{h}^{-1}\big(\Phi(h)\big)\Psi\big(h^{-1}g\big)\right]
\end{equation*}
and the involution $\Phi^{\star_\theta}(g) = \theta_g\!\left[\Phi\big(g^{-1}\big)\right]^*.$ Then $\big(\ell^1_\te(\G;\A),\star_{\theta},^{\star_{\theta}}\!,\|\cdot\|_{\ell^1_\te(\G;\A)}\big)$ is a Banach $^*$-algebra. We will refer to it as {\it the $\ell^1$-partial crossed product}. Its enveloping $C^*$-algebra  is denoted by $\A\!\rtt_\te\!\G$ and called {\it the partial crossed product associated to $(\A,\te,\G)$}\,. We used the fancy symbol $\rtt$ instead of the more usual one $\rtimes$ to stress that our crossed product is assigned to a {\it partial} action. 

\smallskip
We indicate the particularization of Theorem \ref{teoremix} to partial actions. Since $\G$ must be amenable, by \cite{Ex3} the full and the reduced partial crossed products are identical and possess a (faithful) positive contractive conditional expectation. 

\begin{thm}\label{corolyx}
Let $(\A,\te,\G)$ be a partial $C^*$-dynamical system with rigidly symmetric discrete group $\G$ and unital $C^*$-algebra $\A$\,. Then $\ell^1_\te(\G;\A)$ is symmetric and inverse closed in $\A\rtt_\te\!\G$\,. If $\,\Pi:\A\rtt_{\te}\!\G\to \mathbb{B}(\h)$ is a faithful representation, $\Pi\left[\ell^1_\te(\G;\A)\right]$ is inverse closed in $\mathbb{B}(\h)$\,.
\end{thm}

\begin{proof}
The crossed product $\mathfrak C:=\A\!\rtt_\te\!\G$ is a topologically graded $C^*$-algebra \cite{Ex2}, where 
\begin{equation*}\label{verita}
\mathfrak C_g=\big\{\Phi\in\ell^1_\te({\sf G};\A)\,\big\vert\,{\rm supp}(\Phi)\subset\{g\}\big\}\,,\quad\forall\,g\in\G\,.
\end{equation*}
Its $\ell^1$-algebra (with a slight reinterpretation) is
\begin{equation*}\label{veveriz}
\ell^1(\mathfrak C)\equiv\ell^1\big(\A\!\rtt_\te\!\G\big)=\ell^1_\te({\sf G};\A)\,.
\end{equation*}
Thus we can apply Theorem \ref{teoremix}. 
\end{proof}

\begin{rem}\label{apucalaie}
The result can also be obtained directly, by means specific to partial action. To prove the isometric linear embedding $\ell^1_\te({\sf G};\A)\hookrightarrow\ell^1\big(\G;\mathbb B(\h))$ one uses the map 
$\big[T(\Phi)\big]:=\rho\big[\Phi(g)\big]u_g$\,, where $(\rho,u)$ is a covariant representation of $(\A,\te,\G)$ in the Hilbert space $\h$\,, as in \cite[Sect.\ 13]{Ex2}. Here $\rho$ should be a {\it faithful} representation of $\A$ and $u$ a {\it partial} representation of $\G$\,, connected by an equivariance condition.
\end{rem}

\subsection{Inverse closedness modulo ideals.}\label{idealuri}

\begin{defn}\label{idealix}
Let $\mathfrak{K}$ be an ideal of the unital C$^*$-algebra $\mathfrak{C}$\,. The $^*$-subalgebra $\mathfrak{B}$ of $\mathfrak{C}$ is called {\it $\mathfrak{K}$-inverse closed} if $\mathfrak{B}/(\mathfrak{B}\cap\mathfrak{K})$ is inverse closed in $\mathfrak{C}/\mathfrak{K}$\,.
\end{defn}

\begin{ex}\label{fredholmix}
If $\,\mathfrak{K}=\{0\}$\,, the notion coincides with that introduced in Definition \ref{symmetric}. But the main motivating case is as follows: Let $\mathfrak C$ be a $C^*$-algebra of bounded operators in the Hilbert space $\h$\,, containing the ideal $\mathbb K(\h)$ of all the compact operators. We recall that $T\in\mathbb B(\h)$ is called {\it Fredholm} if its range is closed and its kernel and its cokernel are both finite-dimensional. By Atkinson's Theorem, this happens exactly if the image of $T$ in the Calkin algebra $\mathbb B(\h)/\mathbb K(\h)$ is invertible. A $^*$-subalgebra $\mathfrak B$ of $\mathbb B(\h)$ will be called {\it Fredholm inverse closed} if the situation in Definition \ref{idealix} holds with $\mathfrak K=\mathbb K(\h)$\,. For $\mathfrak C$ we can take any $C^*$-algebra of $\mathbb B(\h)$ containing $\mathfrak B$\,.
\end{ex}

\begin{defn}\label{buhonero}
Let $\mathfrak C=\widetilde{\bigoplus}_{g}\mathfrak C_g$ be a graded $C^*$-algebra over the discrete group $\G$\,. 
The ideal $\mathfrak K$ of $\mathfrak C$ is {\it a graded ideal} (also called {\it induced ideal}) if the ideal generated by $\mathfrak K\cap\mathfrak C_\e$ coincides with $\mathfrak K$ (or, equivalently, that $\bigoplus_{g}\mathfrak K\cap\mathfrak C_g$ is dense in $\mathfrak K$).
\end{defn}

\begin{thm}\label{ruye}
Let $\mathfrak C=\widetilde{\bigoplus}_{g}\mathfrak C_g$ be a topologically graded $C^*$-algebra over the discrete rigidly symmetric group $\G$ and $\mathfrak K$ a graded ideal. Then $\ell^1(\mathfrak C)$ is $\mathfrak K$-inverse closed in $\mathfrak C$\,.
\end{thm}

\begin{proof}
One needs to show that $\ell^1(\mathfrak C)/\ell^1(\mathfrak C)\cap\mathfrak K$ is inverse closed in $\mathfrak C/\mathfrak K$\,.
Let us denote by $\kappa:\mathfrak C\to\mathfrak C/\mathfrak K$ the quotient map. By \cite[Prop.\,23.10]{Ex2} (or by \cite[Prop.\,3.11]{Ex3}), under the stated conditions upon the graded algebra and the ideal, $\mathfrak C/\mathfrak K$ is topologically graded by the subspaces 
$$
\big\{(\mathfrak C/\mathfrak K)_g:=\kappa(\mathfrak C_g)\,\big\vert\,g\in\G\big\}\,.
$$ 
The corresponding linear contraction $\widetilde Q_g:\mathfrak C/\mathfrak K\to(\mathfrak C/\mathfrak K)_g$ satisfies
\begin{equation}\label{traznau}
\widetilde Q_g\circ\kappa=\kappa\circ\widetilde P_g\,,\quad\forall\,g\in\G\,.
\end{equation}
Using this and the form of the $\ell^1$-norms, one shows immediately that $\kappa$ sends $\ell^1(\mathfrak C)$ into $\ell^1(\mathfrak C/\mathfrak K)$ contractively. Its kernel is clearly $\ell^1(\mathfrak C)\cap\mathfrak K$\,. Let us show its surjectivity, starting from 
$$
\varphi:=\sum_{g\in\G}\varphi_g\equiv\sum_{g\in\G}\widetilde Q_g(\varphi)\in\ell^1(\mathfrak C/\mathfrak K)
$$ 
(unconditional convergence in $\mathfrak C/\mathfrak K$)\,. Let $\{\beta_g\!\mid\!g\in\G\}$ be a summable family of positive numbers. For each $g$, there is an element $\Phi_g\in\mathfrak C_g$ such that $\kappa(\Phi_g)=\varphi_g$ and $\p\!\Phi_g\!\p_\mathfrak C\,\le\,\p\!\varphi_g\!\p_{\mathfrak C/\mathfrak K}+\beta_g$\,. Then $\Phi:=\sum_{g}\Phi_g\in\ell^1(\mathfrak C)$ and $\kappa(\Phi)=\varphi$\,.
So $\ell^1(\mathfrak C)/\ell^1(\mathfrak C)\cap\mathfrak K$ may be identified with $\ell^1(\mathfrak C/\mathfrak K)$\,. We apply Theorem \ref{teoremix} (ii) to the topologically graded quotient algebra $\mathfrak C/\mathfrak K$ and finish the proof.
\end{proof}

The particular case of partial actions is worth mentioning. Suppose that $(\mathcal{A},\theta,\G)$ is a partial dynamical system with discrete group $\G$ and that $\mathcal{K}$ is an ideal of $\mathcal{A}$ that is $\theta$-invariant: 
$$
\theta_g\big(\A_{g^{-1}}\!\cap\mathcal{K}\big)\subset \mathcal{K}\,,\quad\forall\,g\in \G\,.
$$ 
We denote by the same letter $\theta$ the action of $\G$ by partial automorphisms of $\mathcal{K}$ defined by restrictions. 
One gets the C$^*$-partial dynamical system $(\mathcal{K},\theta,\G)$\,. 
The partial crossed product $\mathcal{K}\rtt_{\theta}\G$ may be identified with an ideal of $\mathcal{A}\rtt_{\theta}\G$\,. Under this identification, $\ell^1_{\theta}(\G;\mathcal{K})$ becomes an ideal of $\ell^1_{\theta}(\G;\mathcal{A})$ in the natural way: the $\ell^1$-function $\Phi:\G\rightarrow \mathcal{K}$ is taken to be $\mathcal{A}$-valued. One can use the exactness of the partial crossed product construction \cite[Sect.\,3]{ELQ} to prove

\begin{thm}\label{theoremix}
Assume that the discrete group $\G$ is rigidly symmetric. The Banach $^*$-algebra $\ell^1_{\theta}(\G;\mathcal{A})$ is $\mathcal{K}\!\rtt_{\theta}\!\G$-inverse closed in the partial crossed product $\mathcal{A}\!\rtt_{\theta}\!\G$ for any $\theta$-invariant ideal $\mathcal{K}$ of $\mathcal{A}$\,.
\end{thm}

Theorem \ref{theoremix} can also be obtained as a consequence of Theorem \ref{ruye}, since partial crossed products are topologically graded, and the ideal $\mathcal{K}\!\rtt_{\theta}\!\G$ is indeed graded under the stated assumptions.

\smallskip
The situation in Exemple \ref{fredholmix} is the most interesting:

\begin{cor}\label{honduras}
Let $(\A,\theta,\G)$ be a partial C$^*\!$-dynamical system with discrete rigidly symmetric group $\G$ and $\K$ a $\theta$-invariant ideal in $\A$\,. Let $\Pi\!:\!\A\rtt_{\theta}\!\G \to \mathbb{B}(\mathcal{H})$ be a faithful representation such that $\Pi\big(\K\rtt_{\theta}\!\G\big)=\mathbb{K}(\mathcal{H})$\,. Then the Banach $^*$-algebra $\Pi[\ell^1_{\theta}(\G;\A)]$ is Fredholm inverse closed.
\end{cor}

\subsection{Other types of decay}\label{decay}

\begin{defn}\label{sconx}
A {\it weight} on the discrete group $\G$ is a function $\nu: \G\to [1,\infty)$ satisfying everywhere
\begin{equation}\label{submultiplicative}
 \nu(gh)\leq \nu(g)\nu(h)\,,\quad\nu(g^{-1})=\nu(g)\,.
\end{equation}
Let $\mathfrak C$ be a topologically graded $C^*$-algebra. On $\,\bigoplus_{g}\!\mathfrak C_g$  we can introduce the norm
\begin{equation}\label{normix}
\p\!\Phi\!\p_{\ell^{1,\nu}(\mathfrak C)}\,:=\sum_{g\in\G}\nu(g)\!\p\!P_g(\Phi)\!\p.
\end{equation}
The completion in this norm, denoted by $\ell^{1,\nu}(\mathfrak C)$\,, can be seen as a Banach $^*$-algebra with the algebraic structure inherited from $\ell^1(\mathfrak C)$ (or from $\mathfrak C$)\,. 
\end{defn} 

\begin{thm}\label{oix}
Let $\G$ be a rigidly symmetric discrete group and $\nu$ a weight. Assume that there exists a generating subset $V$ of $\,\G$ containing the unit $\e$ such that
\begin{enumerate}
\item[(a)] 
the following uGRS (uniform Gelfand-Raikov-Shilov) condition holds:
\begin{equation}\label{limitConditionOfWeight}
\lim_{n\rightarrow \infty}\sup_{g_1,\dots,g_n \in V}\nu(g_1\cdots g_n)^{\frac{1}{n}}=1\,,
\end{equation}
\item[(b)] 
for some finite constant $C$ one has for any $n\in\mathbb{N}$
\begin{equation}\label{limitt}
\sup_{g\in V^n\setminus V^{n-1}}\nu(g) \leq C\!\inf_{g\in V^n\setminus V^{n-1}}\nu(g)\,.
\end{equation}
\end{enumerate}
Then $\ell^{1,\nu}(\mathfrak C)$ is a symmetric Banach $^*$-algebra for every topologically $\G$-graded $C^*$-algebra $\mathfrak C$\,.
\end{thm}

\begin{proof}
The problem of the symmetry of a weighted $\ell^1$-algebra as a $^*$-subalgebra of an unweighted one has been solved in \cite[Th.\,3]{FGL} in a more particular context. We are going to check that the arguments can be adapted to the more general case. The main idea is proving that, given an element $\Phi\in \ell^{1,\nu}(\mathfrak C)$\,, its spectral radius is the same as in $\ell^{1}(\mathfrak C)$\,. Then one applies Hulanicki's Lemma.

\smallskip
As $\nu(\cdot)\geq 1$\,, we have $\p\!\cdot\!\p_{\ell^1(\mathfrak{C})}\,\leq\,\p\cdot\,\p_{\ell^{1,\nu}(\mathfrak{C})}$\,;
the spectral radius formula implies that 
$$
r_{\ell^{1}(\mathfrak{C})}(\Phi)\leq r_{\ell^{1,\nu}(\mathfrak{C})}(\Phi)\,,\quad\forall\,\Phi\in {\ell^{1,\nu}(\mathfrak{C})}\,.
$$
To prove the opposite inequality, pick $V$ a generating set for $\G$ and define on $\mathbb Z$ the function
\begin{equation*}
 v(n):=\sup_{g\in V^{|n|}} \nu(g)\,.
\end{equation*}
Due to the uGRS condition, it is a weight and one has the obvious associated weighted space $\ell^{1,\nu}(\Z)$\,. By induction, from the equation \eqref{normix} we get 
\begin{equation*}
 \p\!\Phi^n \!\p_{\ell^{1,\nu}(\mathfrak{C})}\,\leq \sum_{g_1\in\G}\cdots\sum_{g_n\in\G}\nu(g_1\dots g_n)\p\!P_{g_1}(\Phi) \!\p\cdots \!\p\!P_{g_n}(\Phi) \!\p.
\end{equation*}
Since $\G=\bigsqcup_{m\in\mathbb N}(V^m\setminus V^{m-1})$\,, where $V^0=\emptyset$\,, we may split the sum accordingly. This yields
\begin{equation}\label{shocarix}
\p\!\Phi^n \!\p_{\ell^{1,\nu}(\mathfrak{C})}\,\leq\!\sum_{m_1,\dots,m_n=1}^{\infty}\sum_{V^{m_1}\setminus V^{m_1-1}}\cdots\sum_{V^{m_n}\setminus V^{m_n-1}} \nu(g_1\dots g_n)\p P_{g_1}(\Phi) \!\p\cdots \!\p P_{g_n}(\Phi) \!\p.
\end{equation}
If $g_j\in V^{m_j}\setminus V^{m_j-1}$, then $g_1\cdots g_n \in V^{m_1+\cdots +m_n}$ and so the weight is majorized by
\begin{equation*}
    \nu(g_1\cdots g_n)\leq \sup_{h\in V^{m_1+\cdots +m_n}} \nu(h) =v(m_1+\cdots m_n)\,.
\end{equation*}
Set $b_m:=\sum_{g\in V^m\setminus V^{m-1}}\!\p P_g(\Phi) \!\p$ and $b=(b_m)_{m\in\mathbb N}$\,. Then we have $\!\p\Phi \!\p_{\ell^1(\mathfrak C)}\,=\,\p b \!\p_{\ell^1(\mathbb N)}$\,. Also the condition \eqref{limitt} implies immediately that 
\begin{equation*}
    C^{-1}\!\p b\!\p_{\ell^{1,v}(\mathbb Z)}\,\leq \,\p\Phi \!\p_{\ell^{1,\nu}(\mathfrak C)}\,\leq\,\p\,b \!\p_{\ell^{1,v}(\mathbb Z)}.
\end{equation*}
Returning to \eqref{shocarix} we obtain
\begin{equation*}
    \!\p \Phi^n \!\p_{\ell^{1,\nu}(\mathfrak{C})}\,\leq\!\sum_{m_1,\dots,m_n=1}^{\infty}\!v(m_1+\cdots +m_n)\,b_{m_1}\cdots b_{m_n}\!=\;\p b^n \!\p_{\ell^{1,v}(\mathbb Z)}\,<\infty\,.
\end{equation*}
By its definition the weight $v$ on $\mathbb Z$ satisfies the GRS-condition, and $\ell^{1,v}(\mathbb Z)$ is symmetric by \cite{FGL0}. Hence
\begin{align*}
r_{\ell^{1,\nu}(\mathfrak{C})}(\Phi)&=\lim_{n\rightarrow\infty} \!\p \Phi^n \!\p_{\ell^{1,\nu}(\mathfrak{C})}^{1/n}\leq \lim_{n\rightarrow\infty} \!\p b^n \!\p_{\ell^{1,\nu}(\mathbb{Z})}^{1/n}\\
&=r_{\ell^{1,\nu}(\mathbb{Z})}(b)=r_{\ell^{1}(\mathbb{Z})}(b)\\
&\le\,\p b \!\p_{\ell^{1}(\mathbb{Z})}=\,\p \Phi \!\p_{\ell^{1}(\mathfrak{C})}.
\end{align*}
So for all $k\in\mathbb N$ we have
\begin{equation*}
r_{\ell^{1,\nu}(\mathfrak{C})}(\Phi)=r_{\ell^{1,\nu}(\mathfrak{C})}(\Phi^n)^{1/n}\leq\,\p \Phi^n \!\p_{\ell^{1}(\mathfrak{C})}^{1/n},
\end{equation*}
and by letting $n\rightarrow\infty$ we obtain the required inequality $\,r_{\ell^{1,\nu}(\mathfrak{C})}(\Phi)\leq r_{\ell^{1}(\mathfrak{C})}(\Phi)$\,.
\end{proof}

\begin{rem}\label{labix}
There is a different way to prove Theorem \ref{oix}. We define {\it the Beurling algebra associated to the weight $\nu$ and to the $C^*$-algebra $\mathcal B$} as
\begin{equation*}\label{sharpix}
\ell^{1,\nu}(\G;\mathcal B):=\big\{\Psi\, \vert \,\nu\Psi\in \ell^{1}(\G;\mathcal B)\big\}\,,
\end{equation*}
with norm $\p \cdot \!\p_{\ell^{1,\nu}(\G,\mathcal B)}\,:=\,\p\nu\,\cdot\,\!\p_{\ell^{1}(\G;\mathcal B)}$.
By \cite{FGL}, the assumptions on the weight imply that $\ell^{1,\nu}(\G;\mathcal B)$ is a symmetric Banach $^*$-algebra.
One then shows as in Proposition \ref{isometrix} an isometric embedding $\ell^{1,\nu}(\mathfrak C)\hookrightarrow\ell^{1,\nu}(\G;\mathcal B)$\,, for some $C^*$-algebra $\mathcal B$\,, and use this as in Theorem \ref{teoremix} to deduce the symmetry of $\ell^{1,\nu}(\mathfrak C)$\,.
\end{rem}

\section{Examples}\label{exemple}

\subsection{Inverse closed Banach $^*$-algebras from topological partial actions}\label{restrix}

We introduce now certain partial crossed product $C^*$-algebras associated to partial topological actions of the discrete group. We do not consider the most general case; in particular, we restrict generality in such a way to have certain natural faithful Hilbert space representations. In this way, transparent results on inverse closedness will be available.

\smallskip
So let $(X,\Theta,\G)$ be a {\it partial action} of the group $\G$ (assumed here to be countable) on the compact Hausdorff space $X$. In more detail, it could be denoted by
$\big\{\Theta_g:X_{g^{-1}}\!\to X_g\,\big\vert\, g\in\G\big\}$\,,
where each $X_g$ is an open subset of $X$ and $\Theta_g$ is a homeomorphism with domain $X_{g^{-1}}$ and image $X_g$\,. One requires
\begin{itemize}
\item 
$X_\e=X$ and $\Theta_\e={\rm id}_X$\,,
\item
for every $g,h\in\G$\,, the homeomorphism $\Theta_{gh}$ extends $\Theta_{g}\circ\Theta_{h}$ (this one with maximal domain).
\end{itemize}

In fact each locally compact partial dynamical system may be deduced by restricting a global one in an essentially canonical way, but sometimes the total space $\tilde X$ might fail to be Hausdorff \cite{Ab, Ex2}.

\smallskip
We deduce from $(X,\Theta,\G)$ a partial $C^*$-dynamical system $\big(C(X),\te,\G\big)$\,, where the ideals are 
\begin{equation*}\label{santodomingo}
\mathcal A_g\!:=C\big(X_{g}\big)\equiv\big\{a\in C(X)\,\big\vert\,a(x)=0\,,\,\forall\,x\notin X_{g}\big\}
\end{equation*}
and the isomorphisms are
\begin{equation*}\label{panama}
\te_g:C\big(X_{g^{-1}}\big)\to C\big(X_{g}\big)\,,\quad\te_g(a):=a\circ\Theta_{g^{-1}}\,.
\end{equation*}
Therefore, one can construct the Banach $^*$-algebra $\ell^1_\te\big(\G;C(X)\big)$ and its enveloping partial crossed product $C(X)\!\rtt_\te\!\G$\,. We deduce from Theorem \ref{teoremix}  

\begin{prop}\label{corolix}
If $\,\G$ is rigidly symmetric, $\ell^1_\te\big(\G,C(X)\big)$ is reduced, symmetric and inverse closed in the partial crossed product.
\end{prop}

\smallskip
Assume now that under the partial action  there is an open dense orbit $Y$ with trivial isotropy. Thus, for some (and then any) $y\in Y$ one has
\begin{equation*}\label{guadelupa}
\G(y):=\big\{g\in\G\,\big\vert\,y\in X_{g^{-1}}\,,\Theta_g(y)=y\big\}=\{\e\}\,.
\end{equation*}
Then $X$ will be a regular compactification of the discrete orbit $Y$ and for any $y\in Y$ the orbit map $g\to\Theta_g(y)$ will be a homeomorphism on its image. 

\smallskip
In the Hilbert space $\ell^2\big(Y\big)$ we introduce the representation by multiplication operators
\begin{equation*}\label{mexic}
\pi:C(X)\to\mathbb B\big[\ell^2(Y)\big]\,,\quad\big[\pi(a)\xi\big](y):=a(y)\xi(y)\,.
\end{equation*}
For any $h\in\G\,,\,y\in Y$ and $\xi\in\ell^2(Y)$ one defines 
\begin{equation*}\label{partilix}
\big[u_h(\xi)\big](y):=
\begin{cases}
\xi\big(\Theta_{h^{-1}}(y)\big) &{\rm if}\ \;y\in X_h\,,\\
0\, &{\rm if}\ \;y\notin X_h\,.
\end{cases}
\end{equation*}
Then $\big(\pi,u,\ell^2(Y)\big)$ {\it is a covariant representation of} $\big(C(X),\te,\G\big)$ and {\it the integrated form representation} 
$$
\Pi\equiv\pi\!\rtt u\!:C(X)\!\rtt_\te\!\G\to\mathbb B\big[\ell^2(Y)\big]
$$
acts on $\ell^1_\te\big(\G;C(X)\big)$ as $\Pi(\Phi)=\sum_{h\in\G}\pi\big[\Phi(h)\big]u_h$  and satisfies
\begin{equation}\label{lemikainnen}
\Pi\big(a\otimes\delta_h\big)=\pi(a)u_h\,,\quad\forall\,h\in\G\,,\,a\in C(X_h)\,.
\end{equation}
To express it explicitly, we need some notations. 
For $y,z\in Y$ we set $g_{yz}$ for the unique element of the group such that $\Theta_{g_{yz}}(y)=z$\,; we use the fact that $Y$ is an orbit with trivial isotropy. Note the relations
\begin{equation}\label{dominicana}
g_{yz}^{-1}=g_{zy}\,,\quad g_{zy}=g_{xy}g_{zx}\,,\quad\forall\,x,y,z\in Y.
\end{equation}
A direct computation, relying on \eqref{dominicana}, shows that on $\ell^1_\te\big(\G;C(X)\big)$  the representation $\Pi$ reads 
\begin{equation*}\label{flochax}
\big[\Pi(\Phi)\xi\big](y)=\sum_{z\in Y}\Phi\big(g_{zy},y\big)\xi(z)\,,
\end{equation*}
where we set $\big[\Phi(k)\big](y)=:\Phi(k,y)$\,.

\begin{thm}\label{pohjola}
Suppose that $\G$ is countable and rigidly symmetric and that $Y$ is a dense open orbit of the space $X$, with trivial isotropy. Let $\Phi\in\ell^1_\te\big(\G;C(X)\big)$\,. 
\begin{enumerate}
\item[(i)]
If the operator $\,\Pi(\Phi)$ is invertible, there exists $\Psi\in\ell^1_\te\big(\G;C(X)\big)$ such that $\Pi(\Phi)^{-1}\!=\Pi(\Psi)$\,.
\item[(ii)]
If $\,\Pi(\Phi)$ is Fredholm, at least one of its inverses $T$ modulo $\mathbb K\big[\ell^2(Y)\big]$ belongs to $\Pi[\ell^1_{\theta}\big(\G;C(X)\big)]$\,.
\end{enumerate}
\end{thm}

\begin{proof}
(i) Having in view Proposition \ref{corolix}, we only need to check that the representation $\Pi$ is faithful; see also Theorem \ref{teoremix}(iv). The universal and the reduced partial crossed products here are the same, cf. Remark \ref{lantors}. Note that the restriction of $\Pi$ to $C(X)$ coincides with $\pi$ (set $h=\e$ in \eqref{lemikainnen}). By \cite[Th.\,2.6]{ELQ}, injectivity of $\Pi$ would follow if we prove that $\pi$ is injective and the action is topologically free. The injectivity of $\pi$ is obvious: $\pi(a)=0$ means that the restriction of $a$ to $Y$ is null, which implies that $a$ is null, since the orbit $Y$ was supposed dense. Topological freeness (cf. \cite[Def.\,2.1.]{ELQ}) is clear, since $Y$ is dense: having trivial isotropy, it is disjoint from any set of the form
$$
F_g:=\big\{x\in X_{g^{-1}}\,\big\vert\,\Theta_g(x)=x\big\}\,,\quad g\ne\e\,,
$$
so this one must have empty interior.

\smallskip
(ii) Since $Y$ is discrete, the elements of $C_0(Y)$ are transformed by the representation $\pi$ into compact multiplication operators in $\ell^2(Y)$\,. The $\te$-invariant ideal $C_0(Y)$ gives rise to the ideal $C_0(Y)\!\rightthreetimes_\te\!\G$ of the partial crossed product $C(X)\!\rtt_\te\!\G$ and one has  $\Pi\big[C_0(Y)\!\rtt_\te\!\G\big]=\mathbb K\big[\ell^2(Y)\big]$\,. Therefore one can apply Corollary \ref{honduras} and conclude that $\Pi[\ell^1_{\theta}\big(\G;C(X)\big)]\subset\mathbb B\big[\ell^2(Y)\big]$  is Fredholm inverse closed. Then the assertion follows easily from Atkinson's Theorem; see also Example \ref{fredholmix}.
\end{proof}

\subsection{Toeplitz-Bunce-Deddens operators}\label{dedenix}

In the Hilbert space $\ell^2(\N)$\,, with canonical orthonormal base $\{\delta_n\}_{n\in\N}$\,, each function $a\in\ell^\infty(\N)$ defines a bounded multiplication operator ${\sf M}_a$\,. Let us denote by ${\sf S}$ the right shift uniquely determined by ${\sf S}\,\delta_n\!:=\delta_{n+1}$ for every $n\in\N$\,. For $q\in\N$\,, we say that $a\in\ell^\infty(\N)$ is $q$-periodic if $a(n+q)=a(n)\,,\,\forall\,n\in\N$\,.
We denote by $\ell^\infty(\N;q)$ the $C^*$-subalgebra of all $q$-periodic functions. The starting point in constructing a Toeplitz-Bunce-Deddens algebra is an infinite family $\mathbf q:=\{q_i\!\mid\!i\in\N\}$ of strictly increasing positive integers such that each $q_i$ divides $q_{i+1}$\,, i.e. $q_{i+1}=r_i q_i$ with $r_i\ge 2$\,. 

\begin{defn}\label{buncix}
{\it The Toeplitz-Bunce-Deddens algebra $\mathfrak D(\mathbf q)$ associated to the multi-integer $\mathbf q$} is the $C^*$-algebras of operators in $\ell^2(\N)$ generated by the family
\begin{equation*}\label{generatrix}
\Big\{{\sf S}_a:={\sf S}{\sf M}_a\,\Big\vert\,a\in\bigcup_{i\in\N}\ell^\infty(\N;q_i)\Big\}\,.
\end{equation*}
\end{defn}

\begin{prop}\label{propotrix}
The Toeplitz-Bunce-Deddens algebra is topologically graded over $\Z$\,. Its $\ell^1$-Banach $^*$-algebra $\ell^1\big[\mathfrak D(\mathbf q)\big]$
is reduced, symmetric and inverse closed in $\mathfrak D(\mathbf q)$ and in $\mathbb B\big[\ell^2(\N)\big]$ (an explicit description is included in the proof).
\end{prop}

\begin{proof}
By \cite{Exs1}, a canonical action $\alpha$ of the torus $\mathbb T=\widehat\Z$ on $\mathfrak D(\mathbf q)$ is defined as follows: For every $z\in\mathbb T$ the unitary multiplication operator determined by $U_z\delta_n:=z^n\delta_n$ ($n\in\N$) defines on $\mathbb B\big[\ell^2(\N)\big]$ the inner automorphism $T\to U_z TU_z^*$. Because of
\begin{equation*}\label{costarica}
\alpha_z(S_a)=U_z S_aU_z^*=zS_a\,,\quad\forall\,a\in\ell^\infty(\N)\,,
\end{equation*}
$\big(\mathfrak D(\mathbf q),\alpha,\T\big)$ is indeed a (full) continuous action. We also write $\ell^\infty(\N;\mathbf q)$ for the (unital, Abelian) $C^*$-subalgebra of $\ell^\infty(\N)$ generated by $\bigcup_{i\in\N}\ell^\infty(\N;q_i)$\,.  The spectral subspace $\mathfrak D(\mathbf q)_k$ is generated by operators of the ordered form
\begin{equation*}\label{paracetamol}
S_{a_1}\dots S_{a_n}S^*_{b_1}\dots S^*_{b_m}\,,\quad a_1,\dots,a_n,b_1,\dots,b_m\in\ell^\infty(\N;\mathbf q)\,,\ n-m=k\,.
\end{equation*}
This follows easily from the relations $S^*_a S_b=M_{\overline ab}$ and $M_a S_b=S_{\tilde ab}$\,, where $\tilde a_n:=a_{n-1}$ if $n\ge 1$\,. Since $\tilde a_0$ does not appear in the definition of the weighted shift operator, it can be fixed such that $\tilde a$ is periodic. Other forms of the elements in $\mathfrak D(\mathbf q)_k$ could involve the final projection $Q:=SS^*$.
These having been settled, the assertions concerning the $\ell^1$-Banach algebra follow from Theorem \ref{exelix}.
\end{proof}

\begin{rem}\label{toeplix}
In \cite{Exs1} one gets the isomorphism
\begin{equation}\label{triplix}
\mathfrak D(\mathbf q)\overset{\sim}{\longrightarrow}\big[c_0(\N)\oplus\ell^\infty(\N;\mathbf q)\big]\!\rtt_\te\Z\cong C\big[\N(\mathbf q)\big]\!\rtt_\te\Z\,.
\end{equation}
We refer to \cite{Exs1} for the Gelfand spectrum $\N(\mathbf q)=\N\sqcup\Si(\mathbf q)$ of the Abelian $\alpha$-fixed point $C^*$-algebra 
$$
\mathfrak D(\mathbf q)_0\!:=c_0(\N)\oplus\ell^\infty(\N;\mathbf q)
$$ 
(a compactification of the discrete set $\N$ by a Cantor set $\Si(\mathbf q)$) and for the interesting explicit form of the partial action $\te$ in the Gelfand realization (induced by a partial homeomorphism of $\N(\mathbf q)$). 
\end{rem}

We may also state that {\it $\ell^1\big[\mathfrak D(\mathbf q)\big]=\bigoplus_{k\in\Z}^{1,\alpha}\mathfrak D(\mathbf q)_k$ is Fredholm inverse closed in} $\mathfrak D(\mathbf q)\subset\mathbb B\big[\ell^2(\N)\big]$\,. This follows from Corollary \ref{honduras}, where $\K=c_0(\N)$ and $\Pi$ is the inverse of the isomorphism appearing in \eqref{triplix}, since $\Pi\big[c_0(\N)\!\rtt_\te\!\Z\big]=\mathbb K\big[\ell^2(\N)\big]$\,. One calls the quotient
\begin{equation*}\label{triflix}
\mathfrak D(\mathbf q)/\mathbb K\big[\ell^2(\N)\big]\cong\ell^\infty(\N;\mathbf q)\!\rtimes_\te\Z\cong C\big[\Si(\mathbf q)\big]\!\rtimes_\te\Z
\end{equation*}
{\it the Bunce-Deddens algebra}. The action "at infinity" of $\Z$ on the Cantor set $\Si(\mathbf q)$ is induced by an odometer map. It is a global action, so this quotient can be dealt with by usual crossed products.

\subsection{UHF algebras}\label{afelliz}

{\it An UHF-algebra} is the inductive limit  
\begin{equation}\label{limindiz}
\mathfrak C:=\underset{m\in\N}{\rm lim\,ind}\,\mathbf M^{p_m}
\end{equation} 
of a sequence of full matricial $C^*$-algebras, with unital and injective connecting morphisms, where each $p_m$ divides $p_{m+1}$. In \cite{Exs2} Exel treated general approximately finite algebras. He defined on $\mathfrak C$ a regular and semi-saturated $\mathbb T$-action and then he applied his theory from \cite{Exs} to get an isomorphism between $\mathfrak C$ and a partial crossed product $\A\rtt_\te\Z$\,, where $\A$ is an Abelian almost finite $AF$-algebra. To simplify, we decided to treat only UHF algebras; the general case is similar, but it would involve more complicated notations. To state a result on symmetric subalgebras, we make use of Theorem \ref{exelix}, only invoking the circle action and its spectral subspaces; the partial crossed product will not be mentioned.

\smallskip
On every full matrix algebra $\mathbf M^p$ one defines 
\begin{equation*}\label{inerix}
\alpha^p:\mathbb T\to{\rm Aut}\big(\mathbf M^p\big)\,,\quad\alpha^p_z\big[(c_{ij})_{i,j}\big]=\big(z^{i-j}c_{ij}\big)_{i,j}\,.
\end{equation*}
The connecting morphisms $\big\{\mu^{m}\!:\!\mathbf M^{p_m}\to\mathbf M^{p_{m+1}}\big\}$ are covariant with respect to the actions $\big\{\alpha^{p_m}\big\}$. The inductive limit comes with the canonical monomorphisms $\big\{\nu^{m}\!:\mathbf M^{p_m}\!\to\mathfrak C\big\}$ and $\bigcup_m\nu^{(m)}\!\big(\mathbf M^{p_m}\big)$ is dense in $\mathfrak C$\,. For every $z\in\mathbb T\,,m\in\N$ and $\Phi\in\mathbf M^{p_m}$ we set
\begin{equation*}\label{tetaniz}
\alpha_z\big[\nu^{m}(\Phi)\big]:=\nu^{m}\big[\alpha^{p_m}_z(\Phi)\big]\,.
\end{equation*}
It is shown in \cite[Sect.\,2]{Exs2} that $\alpha$ extends to an action on $\mathfrak C$ (which is semi-saturated and regular).

\smallskip
To make Theorem \ref{exelix} concrete, one needs the spectral subspaces. For each $|k|\le p$\,, let us denote by $\mathbf M^p_k$ the vector subspace of $p\!\times\!p$-matrices with non-empty entries only on the $k'$th diagonal ($c_{ij}=0$ if $i-j\ne k$). If $|k|>p$\,, we set simply $\mathbf M^p_k\!:=\{0\}$\,. Then $\mathbf M^p_k$ is the $k'$th spectral subspace of the action $\alpha^p$. It is follows from the definitions and from the fact that each canonical morphism $\nu^{m}$ is injective that the spectral subspaces of the action $\alpha$ are the closures in $\mathfrak C$ of the unions over $m$ of the images through $\nu^{m}$ of the subspaces $\mathbf M^{p_m}_k$.

\begin{thm}\label{cioflix}
The Banach $^*$-algebra
\begin{equation*}\label{perru}
\ell^1(\mathfrak C)=\bigoplus^{1,\alpha}_{k\in\N}\overline{\bigcup_{m\in\Z}\nu^{m}\big(\mathbf M^{p_m}_k\big)}
\end{equation*}
is reduced, symmetric and inverse closed in $\mathfrak C$\,.
\end{thm}

\begin{rem}\label{takesix}
In \cite{Ta} Takesaki (see also \cite[VIII.17]{FD2}) provided an approach to UHF algebras which would lead to a different grading, and consequently also to a different symmetric $\ell^1$-type Banach $^*$-algebra, this time over the infinite restricted product $\G:=\prod^\prime_{m\in\N}\G_m$ with the discrete topology, where each $\G_m$ is a cyclic group of a certain order $q_m$\,. Actually, the UHF algebra $\mathfrak C$ is isomorphic to an infinite tensor product $\otimes_m\mathbf M^{q_m}$, where $q_m:=p_{m+1}/p_m$\,. In its turn, this one is shown to be isomorphic to the full crossed product $C(\widehat\G\,)\!\rtimes\G$\,. The dual $\widehat\G$ can be identified with the product group $\prod_{m\in\N}\G_m$\,. 
\end{rem}

\subsection{CAR algebras}\label{afellix}

Let $\big(\mathcal R,(\cdot|\cdot)\big)$ be an infinitely dimensional separable Hilbert space and $a:\mathcal R\to\mathbb B(\mathcal H)$ {\it a representation of the canonical anticommutation relations} in the Hilbert space $\h$\,, generating the unital $C^*$-algebra ${\rm CCR}(\mathcal R)\subset\mathbb B(\mathcal H)$\,. Thus $a$ is lineal and for every $r,s\in\mathcal R$ one has
\begin{equation*}\label{anticom}
a(r)a(s)+a(s)a(r)=0\,,\quad a^*(r)a(s)+a^*(s)a(r)=(r|s)\,.
\end{equation*}
It is known that ${\rm CCR}(\mathcal R)$ is isomorphic to the UHF algebra $\mathbf M(2^\infty):=\underset{m\in\N}{\rm lim\,ind}\,\mathbf M^{2^m}$ as well as with the infinite tensor product $\bigotimes_{m\in\N}\mathbf M^2$.

\smallskip
Let also $V:\wG\to\mathbb B(\mathcal R)$ be a strongly continuous unitary representation of the compact Abelian group $\wG$\,. Associated to $V$ there is a continuous action $\upsilon:\wG\to{\rm Aut}\big[{\rm CCR}(\mathcal R)\big]$\,, uniquely defined on generators by
\begin{equation*}\label{creatrix}
\upsilon_\chi\big[a(r)\big]=a\big[V_\chi(r)\big]\,,\quad\forall\,r\in\mathcal R\,,\,\chi\in\wG\,.
\end{equation*}

We denote by $\G$ the dual of $\wG$\,; it is a discrete Abelian group. We already have the grading 
\begin{equation*}\label{carnatix}
{\rm CCR}(\mathcal R)=\widetilde\bigoplus_{g\in\G}{\rm CCR}(\mathcal R)^{\upsilon}_g
\end{equation*}
in terms of spectral subspaces of the action $\upsilon$\,. The next corollary follows directly from Theorem \ref{teoremix}. Note the generality: any representation $\big(\wG,V\big)$ provides a result.

\begin{cor}\label{spermioni}
The corresponding $\ell^1$-algebra 
\begin{equation*}\label{sperioni}
\ell^1\big({\rm CCR}(\mathcal R)\big)=\bigoplus_{g\in\G}^{1,\upsilon}{\rm CCR}(\mathcal R)^{\upsilon}_g
\end{equation*}
is symmetric and inverse closed in ${\rm CCR}(\mathcal R)$ and in $\mathbb B(\mathcal H)$\,. 
\end{cor}

It is clear that $a(r)\in{\rm CCR}(\mathcal R)^{\upsilon}_g$ if and only if $r\in\mathcal R^V_g$, meaning by definition that $V_\chi(r)=\chi(g)r$ for all $\chi\in\wG$\,. Similarly, $a^*(r)\in{\rm CCR}(\mathcal R)^{\upsilon}_g$ if and only if $r\in\mathcal R^V_{g^{-1}}$\,.

\begin{ex}\label{sfarnak}
If $V:\T\to\mathbb B(\mathcal R)$ is given by $V_\tau(r):=e^{2\pi i\tau}r$ and the group duality is implemented by $\tau(k):=e^{2\pi ik\tau}$ ($k\in\Z$)\,, then $\mathcal R^V_1=\mathcal R$\,, while $\mathcal R^V_k=\{0\}$ for any $k\ne 1$\,. Thus $a(r)\in{\rm CCR}(\mathcal R)^{\upsilon}_1$ and $a^*(r)\in{\rm CCR}(\mathcal R)^{\upsilon}_{-1}$\,, from which it follows that
\begin{equation*}\label{followix}
a^*(r_1)\dots a^*(r_n)a(s_1)\dots a(s_m)\in{\rm CCR}(\mathcal R)^{\upsilon}_{m-n}\,,\quad\forall\,r_1,\dots,r_n,s_1,\dots s_m\in\mathcal R\,.
\end{equation*}
\end{ex}

\subsection{Wiener-Hopf algebras associated to quasi-lattice ordered groups}\label{quasilatix}

One can treat this subject by using partial crossed products, as in Subsection \ref{cerculix}, due to results from \cite{QR,ELQ,Ex2}. We found it easier for our purposes to use the initial article \cite{Ni}, combined with Theorem \ref{teoremix}.

\smallskip
We fix a sub-monoid $\pe$ of a discrete amenable group $\G$ such that $\pe\cap\pe^{-1}=\{\e\}$\,. One defines a left-invariant order relation in $\G$ by $g\leq h$ iff $g^{-1}h\in \pe$.
The ordered group $(\G,\pe)$ is {\it quasi-lattice ordered} if any $g\in\pe\pe^{-1}$ has a least upper bound in $\pe$.
Many examples may be found at \cite[pag.23]{Ni}.

\smallskip
Consider the Hilbert space $\ell^2(\pe)$ with its usual orthonormal basis $\{e_q\}_{q\in\pe}$. For $p\in\pe$, consider the bounded linear operator $W_p\in \mathbb{B}(\ell^2(\pe))$ defined by
\begin{equation*}
W_p(e_q)=e_{pq}\,,\quad\forall\,q\in\pe.
\end{equation*}
 We refer to $W$ as the {\it regular semigroup of isometries} of  $\pe$.    
The $C^*$-algebra of operators on $\ell^2(\pe)$ generated by the range of $W$ is called the {\it Wiener-Hopf algebra of $\,\pe$} and denoted by $\mathfrak{W}(\pe)$\,.

\smallskip
For every $g\in\pe\pe^{-1}$ we define the closed subspace
\begin{equation*}
\mathfrak{W}(\pe)_g:=\overline{\textrm{span}}\,\big\{W_p W_q^*\,\big\vert\,pq^{-1}\!=g\big\}
\end{equation*}
and set $\mathfrak{W}(\pe)_g=\{0\}$ if $g\notin \pe\pe^{-1}$.
A characterization of the Abelian $C^*$-subalgebra $\mathfrak{W}(\pe)_\e$ is 
\begin{equation*}
\mathfrak{W}(\pe)_\e =\left\{T\in\mathfrak{W}(\pe)\!\mid\! T \textrm{ has a diagonal matrix relatively to the canonical basis of }\ell^2(\pe)\right\}.
\end{equation*}
In Section 3 of Nica's paper \cite{Ni} it is shown that that {\it the collection $\big\{\mathfrak{W}(\pe)_g\,\big\vert\,g\in\G\big\}$ provides a topologically graded structure of the Wiener-Hopf algebra}. We can apply now Theorem \ref{teoremix} and get

\begin{cor}\label{mimishor}
The Banach $^*$-algebra $\ell^1\Big( \widehat{\bigoplus}_{g\in\G}\,\mathfrak{W}(\pe)_g \Big)$ is symmetric and inverse closed in $\mathbb B\big[\ell^2(\pe)\big]$\,.
\end{cor}

\begin{rem}\label{spresfarsit}
Suppose that $\pe$ is finitely generated (a slightly more general assumption is possible). Then $\ell^1\Big( \widehat{\bigoplus}_{g\in\G}\,\mathfrak{W}(\pe)_g \Big)$ {\it is also Fredholm inverse closed}. We only sketch a proof. As in \cite{QR,ELQ,Ex2}, one has an isomorphism $\mathfrak W(\pe)\cong\mathfrak W(\pe)_\e\rtt_\te\G$ for a certain partial action $\te$ of $\G$ on the diagonal Abelian algebra. This one is induced from a (Bernoulli-type) partial topological action $\Theta$ on its Gelfand spectrum $\Omega_\e$\,, which is a compactification of the monoid $\pe$. This compactification is regular ($\pe$ is an {\it open} dense orbit) under the stated hypothesis on $\pe$. It is shown in \cite[6.3]{Ni} that this regularity is equivalent with the fact that the ideal of all compact operators, identified with $C(\pe)\rtt_\te\G$\,, is contained in $\mathfrak W(\pe)\,.$ Then the result follows easily from our Corollary \ref{honduras}.
\end{rem}

\subsection{Higher rank graph algebras}\label{oroflix}

In this subsection we are going to use constructions from \cite{KP} and \cite{Ra}.

\begin{defn}\label{hix}
{\it A ${\rm k}$-graph} ({\it higher rank graph}) is a countable category $\Lambda$ endowed with a functor ${\rm d}:\Lambda\to\N^{\rm k}$ ({\it the degree functor}) satisfying the {\it factorization property}: For every $\lambda\in\Lambda$ such that ${\rm d}(\lambda)=\mathfrak m+\mathfrak n$\,, there exist unique elements $\mu,\nu\in\Lambda$ such that ${\rm d}(\mu)=\mathfrak m$\,, ${\rm d}(\nu)=\mathfrak n$ and $\lambda=\mu\nu$\,.
\end{defn}

If ${\rm k}=1$\,, one identifies a $1$-graph with the path category $E^*\!:=\bigsqcup_{n\in\N}E^n$ of a directed graph $\big(s,r:E^1\to E^0\big)$ and ${\rm d}(\lambda)$ is the usual length of a path $\lambda:=e_1\dots e_{{\rm d}(\lambda)}$\,. Accordingly, for the ${\rm k}$-graph $\big({\rm d}:\Lambda\to\N^{\rm k}\big)$\,, one denotes by ${\rm s},{\rm r}:\Lambda\to\Lambda^0\equiv{\rm Obj}(\Lambda)$ the source and the range map in the category. One has the decomposition 
$$
\Lambda=\bigsqcup_{\mathfrak n\in\N^{\rm k}}\Lambda^{\mathfrak n}\equiv\bigsqcup_{\mathfrak n\in\N^{\rm k}}{\rm d}^{-1}(\mathfrak n)\,.
$$
For elements from the object set $\Lambda^0$ (called {\it vertices}) one prefers notations as $v,w$\,. We also set set $\Lambda^{\mathfrak n}(v):=\Lambda^{\mathfrak n}\cap{\rm r}^{-1}(v)$ ("paths" of "length" $\mathfrak n\in\N^{\rm k}$ ending in $v\in\Lambda^0$\,). 

\begin{defn}\label{admisibilix}
One says that the ${\rm k}$-graph is {\it admissible} if every $\Lambda^{\mathfrak n}(v)$ is finite (raw-finite ${\rm k}$-graph) and non-void (no sources) and $\Lambda^0$ is finite (insuring that the $C^*$-algebra below is unital, with unit $\sum_{v\in\Lambda^0}{\sf S}_v$).
\end{defn}

\begin{defn}\label{cstarix}
Let $\Lambda\overset{{\rm d}}{\longrightarrow}\N^{\rm k}$ be an admissible ${\rm k}$-graph. {\it A Cuntz-Krieger $\Lambda$-family} is a set $\{{\sf T}_\lambda\!\mid\!\lambda\in\Lambda\}$ of partial isometries in a $C^*$-algebra $\mathfrak D$ such that
\begin{enumerate}
\item[(a)]
the elements $\big\{{\sf T}_v\!\mid\!v\in\Lambda^0\big\}$ are mutually orthogonal projections,
\item[(b)]
if ${\rm s}(\lambda)={\rm r}(\mu)$\,, then ${\sf T}_\lambda{\sf T}_\mu={\sf T}_{\lambda\mu}$\,,
\item[(c)]
${\sf T}^*_\lambda{\sf T}_\lambda={\sf T}_{{\rm s}(\lambda)}$ for every $\lambda\in\Lambda$\,,
\item[(d)]
${\sf T}_v=\sum_{\lambda\in\Lambda^{\mathfrak n}(v)}{\sf T}_\lambda{\sf T}^*_\lambda$ for every $v\in\Lambda^0$ and $\mathfrak n\in\N^{\rm k}$.
\end{enumerate}
We define $C^*(\Lambda,{\rm d})\equiv\mathfrak C(\Lambda)$ to be the universal $C^*$-algebra generated by a Cuntz-Krieger $\Lambda$-family $\{{\sf S}_\lambda\!\mid\!\lambda\in\Lambda\}$\,. Universality means that for every Cuntz-Krieger $\Lambda$-family $\{{\sf T}_\lambda\!\mid\!\lambda\in\Lambda\}\subset\mathfrak D$\,, there exists a unique $C^*$-morphism $\pi_{\sf T}:\mathfrak C(\Lambda)\to\mathfrak D$ such that $\pi_{\sf T}({\sf S}_\lambda)={\sf T}_\lambda$ for every $\lambda$\,.
\end{defn}

We identify now families of symmetric Banach $^*$-algebras, starting with functors ${\sf F}:\Lambda\to\G$\,, where $\G$ is supposed to be a discrete Abelian group. One particular case is $\G=\N^{\rm k}$ and ${\sf F}={\rm d}$\,, but many others are possible. For every $g\in\G$ one sets
\begin{equation*}\label{bebex}
\mathfrak C(\Lambda)_g:=\overline{\rm span}\big\{{\sf S}_\lambda{\sf S}^*_\mu\,{\big\vert\,\sf F}(\lambda){\sf F}(\mu)^{-1}=g\big\}\,.
\end{equation*}
It turns out that these are spectral subspaces of a dual action $\alpha^{\sf F}:\wG\to{\rm Aut}\big[\mathfrak C(\Lambda)\big]$\,,
uniquely determined (use the universal property of $\mathfrak C(\Lambda)$) by
$$
\alpha^{\sf F}_\chi({\sf S}_\lambda)=\chi\big[{\sf F}(\lambda)\big]{\sf S}_\lambda\,,\quad\forall\,\lambda\in\Lambda\,,\,\chi\in\wG\,.
$$

Applying Theorem \ref{exelix} one gets

\begin{cor}\label{ix}
Let $\Lambda\overset{{\rm d}}{\longrightarrow}\N^{\rm k}$ be an admissible ${\rm k}$-graph and ${\sf F}:\Lambda\to\G$ a functor with values in an Abelian discrete group. Then
\begin{equation*}\label{peru}
\ell^1\big(\mathfrak C(\Lambda)\big):=\bigoplus_{g\in\G}^{1,\alpha^{\sf F}}\,\overline{\rm span}\big\{{\sf S}_\lambda{\sf S}^*_\mu\,{\big\vert\,\sf F}(\lambda){\sf F}(\mu)^{-1}=g\big\}
\end{equation*}
is a reduced and symmetric Banach $^*$-algebra, which is inverse closed in $\mathfrak C(\Lambda)$\,.
\end{cor}


\begin{defn}\label{aperiodix}
On $\N^{\sf k}$ one considers the order $\mathfrak m\le\mathfrak n\,\Leftrightarrow\,\mathfrak m_i\le\mathfrak n_i\,,\forall\,i$\,. Let $\Lambda\overset{{\rm d}}{\longrightarrow}\N^{\rm k}$ be an admissible ${\rm k}$-graph.
\begin{enumerate}
\item[(a)]
We say that $\rho$ is {\it a minimal common extension} of $\mu,\nu\in\Lambda$\,, and we write $\rho\in{\rm MCE}(\mu,\nu)$\,, if ${\rm d}(\rho)=\max\{{\rm d}(\mu),{\rm d}(\nu)\}$ and $\rho=\mu\mu'=\nu\nu'$ for some $\mu',\nu'\in\Lambda$\,.
\item[(b)]
The ${\sf k}$-graph is called {\it aperiodic} if for every $\mu,\nu\in\Lambda$ with ${\rm s}(\mu)={\rm s}(\nu)$ there exists $\lambda\in\Lambda$ with ${\rm r}(\lambda)={\rm s}(\mu)$ and ${\rm MCE}(\mu\lambda,\nu\lambda)=\emptyset$\,.
\end{enumerate}
\end{defn}

The two notions are easier to visualize for ${\sf k}=1$\,. In this case, for instance, ${\rm MCE}(\mu,\nu)$ is either void, or a singleton (one of the two elements) and aperiodicity boils down to the well-known condition (L).

\begin{prop}\label{perdrix}
Let $\Lambda\overset{{\rm d}}{\longrightarrow}\N^{\rm k}$ be an admissible and aperiodic ${\rm k}$-graph and ${\sf F}:\Lambda\to\G$ a functor with values in an Abelian discrete group. Let $\{{\sf T}_\lambda\!\mid\!\lambda\in\Lambda\}$ be a Cuntz-Krieger $\Lambda$-family of partial isometries in the Hilbert space $\h$ such that ${\sf T}_v\ne 0$ for every $v\in\Lambda^0$. Then (the quite obvious action $\beta^{\sf F}$ is explained in the proof)
\begin{equation}\label{fujimori}
\bigoplus_{g\in\G}^{1,\beta^{\sf F}}\,\overline{\rm span}\big\{{\sf T}_\lambda{\sf T}^*_\mu\,{\big\vert\,\sf F}(\lambda){\sf F}(\mu)^{-1}\!=g\big\}
\end{equation}
is inverse closed in $\mathbb B(\h)$\,.
\end{prop}

\begin{proof}
Under the stated assumptions, "the Cuntz-Krieger unicity theorem" (see \cite{KP,Si} for example) states that the unique representation $\pi_{\sf T}:\mathfrak C(\Lambda)\to\mathbb B(\h)$ satisfying
$$
\pi_{\sf T}({\sf S_\lambda})={\sf T}_\lambda\,,\quad\forall\lambda\in\Lambda
$$
(provided by the universal property) is injective. We denote by $\mathfrak C({\sf T})$ its range, so that $\mathfrak C(\Lambda)$ and $\mathfrak C({\sf T})$ are isomorphic. Theorem \ref{teoremix} states then that $\pi_{\sf T}\Big[\ell^1\big(\mathfrak C(\Lambda)\big)\Big]$ is inverse closed in $\mathfrak C({\sf T})$ and in $\mathbb B(\h)$\,. There is an action $\beta^{\sf F}\!:\wG\to{\rm Aut}\big[\mathfrak C({\sf T})\big]$ such that 
$$
\beta^{\sf F}_\chi({\sf T}_\lambda)=\chi[{\sf F}(\lambda)]{\sf T}_\lambda\,,\quad\forall\,\chi\in\wG\,,\,\lambda\in\Lambda
$$
and $\pi_T$ intertwines the two actions. It follows that $\pi_{\sf T}\Big[\ell^1\big(\mathfrak C(\Lambda)\big)\Big]$ coincides with \eqref{fujimori}.
\end{proof}


\medskip
D. Jaur\'e and M. M\u antoiu:

\smallskip
Facultad de Ciencias, Departamento de Matem\'aticas, Universidad de Chile

Las Palmeras 3425, Casilla 653\,, Santiago, Chile.

E-mails: diegojaure@ug.uchile.cl\, and \,mantoiu@uchile.cl


\begin{thebibliography}{1}

\bibitem{Ab} F. Abadie: \emph{Envelopping Actions and Takai Duality for Partial Actions}, J. Funct. Anal. \textbf{197}, 14--67, (2003).

\bibitem{BB} I. Beltita and D. Beltita: {\it Inverse-Closed Algebras of Integral Operators on Locally Compact Groups}, Ann. H. Poincar\'e, \textbf{16}, 1283--1306, (2015).


\bibitem{BO} N. P. Brown and N. Ozawa: \emph{$C^*$-Algebras and Finite-Dimensional Approximations}, Graduate Studies in Mathematics, \textbf{88}, American Mathematical Society, Providence, Rhode Island, 2008.


\bibitem{Exs} R. Exel: \emph{Circle Actions on C*-Algebras, Partial Automorphisms, and a Generalized Pimsner-Voiculescu Exact Squence}, J. Funct. Anal. \textbf{122}, 361--401, (1994).

\bibitem{Exs1} R. Exel: \emph{The Bunce-Deddens Algebras as Crossed Products by Partial Automorphisms}, Bull. Braz. Math. Soc. (N.S.), \textbf{25}, 173--179, (1994).

\bibitem{Exs2} R. Exel: \emph{Approximately Finite $C^*$-algebras and Partial Automorphisms}, Math. Scand. \textbf{77}, 281--288, (1995).

\bibitem{Ex1} R. Exel: \emph{Twisted Partial Actions: a Classification of Regular $C^*$-Algebraic Bundles}, Proc. London Math. Soc. \textbf{74}(3), 417--443, (1997).

\bibitem{Ex3} R. Exel: \emph{Amenability for Fell Bundles}, J. Reine Angew. Math.\textbf{492}, 41--73, (1997).

\bibitem{Ex2} R. Exel: \emph{Partial Dynamical Systems, Fell Bundles and Applications}, Mathematical Surveys and Monographs, \textbf{224}, 2017.

\bibitem{ELQ} R. Exel, M. Laca and J. Quigg: \emph{Partial Dynamical Systems and $C^*$-Algebras Generated by Partial Isometries}, J. Oper. Th., \textbf{47}(1), 169--186, (2002).

\bibitem{FD2} J. M. G. Fell and R. S. Doran: \emph{Representations of $^*$-Algebras, Locally Compact Groups, and Banach $^*$-Algebraic Bundles}, Pure and Applied Mathematics, \textbf{126}, Academic Press, 1988.

\bibitem{FL} G. Fendler and M. Leinert: \emph{On Convolution Dominated Operators}, Integral Equations Operator Theory, \textbf{86}(2), 209--230, (2016).

\bibitem{FGL0} G. Fendler, K. Gr\"ochenig and Leinert: \emph{Symmetry of Weighted $L^1$-Algebras and the GRS-Condition}, Bull. London Math. Soc. \textbf{38}, 625--635, (2006).

\bibitem{FGL} G. Fendler, K. Gr\"ochenig and M. Leinert: \emph{ Convolution-Dominated Operators on Discrete Groups}, Integral Equations Operator Theory, \textbf{61}(4), 490--509, (2008).


\bibitem{Gr} K. Gr\"ochenig: \emph{Wiener's Lemma: Theme and Variations. An Introduction to Spectral Invariance}, in B. Forster and P. Massopust, editors, Four Short Courses on Harmonic Analysis, Appl. Num. Harm. Anal. Birkh\"auser, Boston, 2010.

\bibitem{GL1} K. Gr\"ochenig and M. Leinert: \emph{Wiener's Lemma for Twisted Convolution and Gabor Frames}, J. Amer. Math. Soc. \textbf{17}, 1--18, (2004).

\bibitem{GL2} K. Gr\"ochenig and M. Leinert: \emph{Symmetry and Inverse-Closedness of Matrix Algebras and Functional Calculus for Infinite Matrices}, Trans. of the A.M.S. \textbf{358}(6), 2695--2711, (2006).

\bibitem{KP} A. Kumjian and D. Pask: \emph{Higher Rank Graph $C^*$-algebras}, New York J. Math. \textbf{6}, 1--20, (2000).



\bibitem{Ku} W. Kugler: \emph{On the Symmetry of Generalized $L^1$-Algebras}, Math. Z. \textbf{168}(3), 241--262, (1979).

\bibitem{LP} H. Leptin and D. Poguntke: \emph{Symmetry and Nonsymmetry for Locally Compact Groups}, J. Funct. Anal. \textbf{33}(2), 119--134, (1979).

\bibitem{Ma} M. Mantoiu: \emph{Symmetry and Inverse Closedness for Banach $C^*$-Algebras Associated to Discrete Groups}, Banach J. Math. Anal. \textbf{9}(2), 289--310, (2015).

\bibitem{MC} K. McClanahan: \emph{$K$-Theory for Partial Crossed Products by Discrete Groups}, J. Funct. Anal. \textbf{130}, 77--117, (1995).

\bibitem{Ni} A. Nica: \emph{$C^*$-Algebras Generated by Isometries and Wiener-Hopf Operators}, J. Oper. Th. \textbf{27}(1), 17--52, (1992).


\bibitem{Pa2} T.W. Palmer: \emph{Banach Algebras and the General Theory of $^*$-Algebras}, Vol. II. $^*$-Algebras, Encyclopedia of Mathematics and its Applications, \textbf{69}. Cambridge University Press, Cambridge, 2001.

\bibitem{Po} D. Poguntke: \emph{Rigidly Symmetric $L^1$-Group Algebras}, Seminar Sophus Lie, \textbf{2}, 189--197, (1992).                

\bibitem{Qu} J. C. Quigg: \emph{Discrete $C^*$-Coactions and $C^*$-Algebraic Bundles}, J. Austral. Math. Soc. (Series A) \textbf{60}, 204-221, (1996).

\bibitem{QR} J. C. Quigg and I. Raeburn: \emph{Characterizations of Crossed Products by Partial Actions}, J. Oper. Th. \textbf{37}, 311--340, (1997).

\bibitem{Ra} I. Raeburn: \emph{Deformations of Fell Bundles and Twisted Graph Algebras}, Math. Proc. Cambridge Philos. Soc. \textbf{161}, 535--558, (2016).

\bibitem{Ra1} I. Raeburn: \emph{On Graded $C^*$-Algebras}, Bull. Aust. Math. Soc. \textbf{97}, 127--132, (2018). 

\bibitem{SW} E. Samei and M. Wiersma: \emph{Quasi-Hermitian Locally Compact Groups are Amenable}, Adv. Math. \textbf{359}, 106897, (2020).

\bibitem{Si} A. Sims: \emph{Lecture Notes on Higher-Rank Graphs and Their $C^*$-Algebras}.

\bibitem{Ta} M. Takesaki: \emph{A Liminal Crossed Product of a Uniformly Hyperfinite $C^*$-Algebra by a Compact Abelian Automorphism Group}, J. Funct. Anal. \textbf{7}, 140--146, (1971).

\bibitem{Wie} N. Wiener: \emph{Tauberian Theorems}, Ann. of Math.  \textbf{33}(1), 1--100, (1932).

\end{thebibliography}
\end{document}